\documentclass{amsart}

\usepackage{pdfpages}
\usepackage{amsmath,amssymb,amsthm}
\usepackage{hyperref}
\usepackage{amsmath,amsfonts,amssymb,amsthm}
\usepackage{mathtools}
\usepackage{amsmath}
\usepackage{amsthm}
\usepackage{amssymb}
\usepackage{mathtools}
\usepackage[all]{xy}
\usepackage[normalem]{ulem}
\usepackage{enumitem}
\usepackage{comment}
\newtheorem{Conjecture}{Conjecture}

\newtheorem{lem}{Lemma}[section]
\newtheorem{teo}[lem]{Theorem}
\newtheorem{pro}[lem]{Proposition}

\newtheorem{claim}[lem]{Claim}
\newtheorem{rem}[lem]{Remark}

\newtheorem*{Claim*}{Claim}
\newtheorem*{rem*}{Remark}
\newtheorem*{teo*}{Theorem}

\newcounter{claimcounter}
\numberwithin{claimcounter}{lem}

\newcommand{\OO}{\mathcal O}

\DeclareMathOperator{\proj}{proj}

\DeclareMathOperator{\Fix}{Fix}
\DeclareMathOperator{\Tr}{Tr}

\DeclareMathOperator{\GL}{GL}

\DeclareMathOperator{\im}{Im}

\DeclareMathOperator{\rk}{rk}

\DeclareMathOperator{\Stab}{Stab}

\DeclareMathOperator{\Mat}{Mat}

\newcommand{\R}{\mathbb{R}}

\newcommand{\F}{\mathbb{F}}

\newcommand{\N}{\mathbb{N}}
\newcommand{\CC}{\mathbb{C}}
\newcommand{\Q}{\mathbb{Q}}

\usepackage{hyperref}

\date{April 2023}


\title{Twisted $L^2$-Betti numbers for sofic groups}
\author{ Jan Boschheidgen}
\author{Andrei Jaikin-Zapirain}

\address{Departamento de Matem\'aticas, Universidad Aut\'onoma de Madrid \and  Instituto de Ciencias Matem\'aticas, CSIC-UAM-UC3M-UCM}
\email{jan.boschheidgen@uam.es}
\email[Corresponding author]{andrei.jaikin@uam.es}

\begin{document}

	\begin{abstract}  For a given group $G$, Wolfgang L\"uck asked whether twisting a chain complex of finitely generated free
		$\CC[G]$-modules with a finite dimensional complex representation $V$ of $G$ before passing
		to the $L^2$-completion has no other effect on $L^2$-Betti numbers than a scaling by the
		factor $\dim_{\CC} V$. The purpose of the article  is to answer this question
		affirmatively provided $G$ is sofic. \end{abstract}
	
	\maketitle

	\section{Introduction}
	Given a group $G$ we denote by $\rk_G$ the von Neumann Sylvester matrix rank function of the group algebra  $\CC[G]$. Although $\rk_G$ is defined in an analytic way (we recall its definition in Subsection \ref{prelim:l2}), it may be characterized algebraically for some groups (see for, example, the case of locally indicable groups \cite[Corollary 6.2]{JL20}). Therefore,  we think that it is  plausible that in general there is also a pure algebraic way to define $\rk_G$.  If this is true, then $\rk_G$ should be ``rigid" under natural algebraic manipulations. For example, in \cite{Ja19}, the following conjecture was proposed.
	\begin{Conjecture}[The independence conjecture]\label{indep} Let $G$ be a  group. Let $K$ be a field and let $\phi_1,\phi_2:K\to \CC$ be  two embeddings of $K$ into $\CC$. Then for every  matrix $A\in \Mat_{n\times m}(K[G]) $ 
		$$\rk_G(\phi_1(A))=\rk_G(\phi_2(A)).$$
	\end{Conjecture}
	This conjecture was proved for sofic groups in \cite{Ja19} and for locally indicable groups in \cite{JL20}.
	
	In this paper we consider a similar problem. Given a representation $\sigma:G\to \GL_k(\CC)$, we   define $\tilde \sigma: \CC[G]\to \Mat_k(\CC[G])$ by sending $g\in G$ to $\sigma(g)g$ and then extending by linearity. The following conjecture is a rephrasing of a question raised by L\"uck in \cite[Question 0.1]{Lu18}.
	\begin{Conjecture}[The L\"uck twisted conjecture]\label{twist} Let $G$ be a  group and $\sigma:G\to \GL_k(\CC)$ a homomorphism.  Then for every  matrix $A\in \Mat_{n\times m}(\CC[G])$
		$$\rk_G( \tilde \sigma(A))=k\cdot  \rk_G(A).$$
	\end{Conjecture}
	L\"uck noticed that a positive solution of this conjecture for a group $G$ would allow to calculate $L^2$-Betti numbers of some fibrations of connected finite $CW$-complexes $F\to E\to B$ where  $\pi_1(B)\cong G$ and the map $\pi_1(E)\to \pi_1(B)$ induced by the fibration is an isomorphism.  More details on this application can be found in \cite[Subsection 5.2] {Lu18} and \cite[Section 4]{KS21}. Conjecture \ref{twist} was proved for torsion-free elementary amenable groups in \cite{Lu18} and for locally indicable groups in \cite{KS21}. In our main result  we prove the conjecture for sofic groups.
	
	\begin{teo} \label {twistedsofic}  The L\"uck twisted conjecture holds for sofic groups.  \end{teo}
	We introduce the notion of sofic group in Subsection \ref{prelim:sofic}. In particular,   amenable and residually finite groups are sofic.

	Notice that the sofic case of  the independence conjecture follows  
	immediately from the sofic L\"uck  approximation  proved in \cite{Ja19} (see Theorem \ref{soficluck}). 
	However, in the case of the L\"uck twisted conjecture, this implication is not so direct.

	Let us outline the   proof of  Theorem \ref{twistedsofic}.   The equality is shown by approximation techniques in two different
	ways: the multiplication operator is approximated using the ``almost" $G$-sets of a
	sofic approximation of $G$ and by approximating the complex matrix coefficients by
	algebraic numbers.
	
	In more detail, we assume first  that  $$\sigma:G\to  \GL_k(\overline \Q) \textrm{\ and\ } A\in \Mat_{n\times m}(\overline \Q[G]).$$
	We may moreover assume that $G$ is finitely generated so that  the image of $\sigma$ is determined by finitely
	many matrices which together with $A$ have all coefficients in a ring of $S$-integers  $\mathcal O_{K,S}$  in some number field $K$. This ring can be approximated by finite fields
	$\F_i = \mathcal O_{K,S}/P_i$ where $P_i$ runs trough different maximal ideals (so that necessarily
	$|\F_i|\to \infty$). An arbitrary sofic approximation $\{X_i\}$ of $G = F/N$ is now replaced by a sofic
	approximation $\{Y_i\}$ such that any $f\in F$ fixing a point of $Y_i$ also lies in the kernel of
	the composition $$F\to G\to  \GL_k(O_{K,S})\to \GL_k(\F_i).$$  Consider a matrix  $B$   over $\Q[F]$ that maps to  $A$ under the canonical homomorphism $F\to G$. 
	By abuse of notation we denote by $\sigma$ also the composition $F \to G \to \GL_k(O_{K,S})$ and by $\tilde{\sigma}$ the map
		$\CC[F] \to \Mat_k(\CC[F])$ defined as before for $\CC[G]$. It is an elementary
	consideration that
	$$\rk_{Y_i, P_i} (\widetilde \sigma (B)) = k\cdot  \rk_{Y_i, P_i} (B)$$
	holds for the associated finite analogues $\rk_{Y_i, P_i}$. Using previously developed
	methods from \cite{Ja19} and the sofic L\"uck approximation over $\overline \Q$, we
	show that the equality passes to $\rk_G$ instead of $\rk_{Y_i, P_i}$.
	
	For general complex coefficients, we  specialize any possibly occurring transcendental
	elements to nearby algebraic numbers with the same algebraic dependencies
	and show that the $\rk_G$ is ``continuous" with respect to this operation. To do so,  we use that $\overline \Q$-points lie dense (in the Euclidean topology) in complex algebraic
	varieties defined over $\overline \Q$, and deduce an approximation statement for $\rk_G$ from
	weak convergence of spectral measures, similarly as one does to obtain Kazhdan's
	inequality. The final argument involves the sofic L\"uck approximation over $\CC$.

	\subsection*{Acknowledgments} This paper is partially supported by the grants  of the Ministry of Science and Innovation of Spain (reference numbers MTM2017-82690-P and PID2020-114032GB-I00)   and by the ICMAT Severo
	Ochoa project  CEX2019-000904-S4.  We are deeply grateful to an anonymous referee for enhancing the introduction of the article and providing valuable suggestions. Additionally, we would like to express our thanks to Holger Kammeyer for bringing the reference \cite{BG08} to our attention.
	
	\section{Preliminaries}
	
	\subsection{General notation} 
	The linear operators on vector spaces will act on the right. If $\mathcal H$ is a Hilbert space, $U(\mathcal H)$  denotes the group of unitary operators on $\mathcal H$ and $\mathcal B(\mathcal H)$ is the algebra of bounded operators.
	
	In this paper the action of a group on a set is on the right side. If $F$ acts on $X$ and $w\in F$ then  $\Fix_{X}(w)$ denotes the set of fixed points of $w$ in $X$.
	
	We will use the language of Sylvester matrix rank functions. They are maps from the set of matrices over a ring to the set of non-negative real numbers. For a precise definition, see \cite[Section 5]{Ja19b}.
	
	\subsection{Von Neumann Sylvester matrix rank function}\label{prelim:l2}
	In this subsection we recall the definition of the von Neumann Sylvester matrix rank function $\rk_G$ on $\CC[G]$. 
	
	First, consider $G$ to be countable. Let $\rho_G:G\to U(l^2(G))$  and $\lambda_G: G\to U(l^2(G))$ be the right and left  regular representation of $G$, respectively: 
	$$\left (\sum_{h\in G} a_h h \right )\rho_G(g)=\sum_{h\in G} a_h hg,\ \left (\sum_{h\in G} a_h h \right )\lambda_G(g)=\sum_{h\in G} a_h g^{-1}h\ (a_h\in \CC, g\in G).$$
	By linearity we can extend $\rho_G$ and $\lambda_G$ to the homomorphisms $$\rho_G,\lambda_G:  \CC[G]\to \mathcal B(l^2(G)).$$
	
	A  finitely generated {\bf Hilbert}  $G$-module is a closed subspace $V\le l^2(G)^n$,  which is invariant under the   actions of  elements of $\lambda_G(G)$.  We denote by  $\proj_{V}: l^2(G)^n\to  l^2(G)^n$ the orthogonal projection onto $V$ and we define $$\dim_{G}V:=\Tr_{G}(\proj_{V}):=\sum_{i=1}^n\langle ( \mathbf 1_i)\proj_{V},\mathbf 1_i\rangle_{(l^2(G))^n},$$
	where $\mathbf 1_i $ is the element of $ l^2(G)^n$ having 1 in the $i$th entry and $0$ in the rest of the entries.     The number $\dim_{G}V $ is the {\bf von Neumann dimension} of $V$.  
	
	Let   $A \in \Mat_{n\times m}(\CC[G])$ be a matrix over $\CC[G]$. 
	The action of $A$ by  right multiplication on $l^2(G)^n$  induces   a bounded linear operator  $\rho_{G}(A):l^2(G)^n\to l^2(G)^m$.  Now we can define the {\bf von Neumann Sylvester matrix rank function} $\rk_G$: 
	$$\rk_G(A):=\dim _G \overline{\im   \rho_{G}(A)}=n-\dim _G \ker  \rho_{G}(A) .$$
	If $G=F/N$ is a quotient of a group $F$ and $A\in \Mat_{n\times m}(\CC[F])$ is a matrix over $\CC[F]$, by abuse of notation, we will also write $\rk_G(A)$ instead of $\rk_G(\overline A)$, where $\overline A$ is the image of $A$ in $\Mat_{n\times m}(\CC[F/N])$. 
	
	If $G$ is not a countable  group then $\rk_G$ is defined as follows. Take a matrix $A$ over $\CC[G]$. Then the group elements that appear in $A$ are contained in  a finitely generated group $H$. We will put $\rk_G(A)=\rk_H(A)$. One easily checks that the value $\rk_H(A)$ does not depend on the subgroup $H$.
	
	\subsection{Sofic groups and the sofic L\"uck approximation}
	\label{prelim:sofic}

	Let  $F$ be a free finitely generated group and assume that it is freely generated by a set $S$. 
	An element $w$ of $F$ has {\bf length} $n$ if $w$ can be expressed as a product of $n$ elements from $S\cup S^{-1}$ and $n$ is the smallest number with this property. We denote by $B_n(1)$ the set of  elements of length at most $n$.
	
	Let $N$ be a normal subgroup of $F$. We put $G=F/N$. We say that $G$ is {\bf sofic} if there is a family $\{X_k:\ k\in\N\}$ of  finite $F$-sets   such that for any $w\in F$,
	$$\displaystyle \lim_{k\to \infty} \frac{|\Fix_{X_k}(w)|}{|X_k|}=1 \textrm{\ if \ } w\in N \textrm{\ and \ }\displaystyle \lim_{k\to \infty}  \frac{|\Fix_{X_k}(w)|}{|X_k|}=0  \textrm{\  if \ } w\not \in N.$$ 
	The  family of $F$-sets $\{X_k\}$  is called a {\bf sofic approximation} of $F/N$.  For an arbitrary group $G$ we say that $G$ is {\bf sofic} if every   finitely generated subgroup of $G$ is sofic. Amenable groups and residually finite groups are sofic. At this moment,  no nonsofic group is known.
	
	Let   $B \in \Mat_{n\times m}(\CC[F])$ be a  matrix over $\CC[F]$. 
	By   multiplication on the right  side, $B$ induces   a  linear operator  $\rho_{X_k}(B):l^2(X_k)^n\to l^2(X_k)^m$. We put 
	$$ \rk_{X_k}(B):=\frac{\dim_{\CC} \im \rho_{X_k}(B)}{|X_k|}.$$
	The sofic L\"uck approximation is the following result.
	\begin{teo} \cite[Theorem 1.3]{Ja19} \label{soficluck}   Let $\{X_k\}$ be a sofic approximation of $G=F/N$. Then for every  $B\in  \Mat_{n\times m}( \CC[F])$,   
		$$\displaystyle \lim_{k\to\infty} \rk_{X_k}(B)=\rk_{G}(B).$$   \end{teo}
	
	This result has its origin in the paper of L\"uck \cite{Lu94} (see \cite{Ja19b} for more details). 
		
	\section{The proof of the main result}
	
	In this section we  prove Theorem \ref{twistedsofic}. As we have explained in the introduction we consider first the case where the matrix $A$ and the homomorphism $\sigma$ are realized over algebraic numbers. This will be done in the first subsection. In the second subsection we prove the general case.
	
	\subsection{The algebraic case}  In this subsection we prove the following result.
	\begin{teo}\label{algebraic}
		Let $G$ be a  sofic group and $\sigma:G\to \GL_k(\overline \Q)$ a homomorphism.  Then for every  matrix $A\in \Mat_{n\times m}(\overline \Q[G])$
		$$\rk_G( \tilde \sigma(A))=k\cdot  \rk_G(A).$$
	\end{teo}
	First we need to prove an auxiliary result. 
	Let  $Y$ be a finite right $F$-set, $\F$  a field and $V$   a right $\F[F]$-module of   dimension $k$ over $\F$. Denote two structures of $\F[F]$-module on the $F$-space $V\otimes_{\F}\F[Y]$:
	
	$$(v\otimes y)\cdot_1 f=v\otimes (y\cdot f) \textrm {\ and\ } (v\otimes y)\cdot_2 f=(v\cdot f)\otimes (y\cdot f),\ (v\in V, y\in Y, f\in F).$$
	Denote these two modules by $(V\otimes_{\F}\F[Y])_1$ and $(V\otimes_{\F}\F[Y])_2$ respectively. The following lemma  resembles \cite[Lemma 1.1]{Lu18}.
	\begin{lem}\label{isomtwo}
		Assume that for any $y\in Y$, $v\in V$ and  $f\in \Stab_F(y)$,  $vf=v$. Then
		$$(V\otimes_{\F} \F[Y])_2\cong (V\otimes_{\F} \F[Y])_1\cong \F[Y]^k \textrm{\ as $\F[F]$-modules.}$$
	\end{lem}
	\begin{proof} It is clear that $ (V\otimes_{\F} \F[Y])_1\cong \F[Y]^k$.
		Without loss of generality we can assume that $F$ acts transitively on $Y$. Hence $Y=x\cdot F$ for some $x\in Y$.  Define the following $\F$-map
		$$\alpha: (V\otimes_{\F} \F[Y])_1\to (V\otimes_{\F} \F[Y])_2, \ \alpha(v\otimes x\cdot f)=v\cdot f\otimes x\cdot f \ (v\in V,   f\in F).$$
		Observe that, since  $\Stab_F(x)$ acts trivially on $V$, $\alpha$ is well defined. On the other hand, if $g\in F$ we obtain
		$$\alpha((v\otimes x\cdot f)\cdot _1g)=\alpha(v\otimes x\cdot  fg)=v\cdot fg\otimes x\cdot fg=(v\cdot f\otimes x\cdot f)\cdot _2g=\alpha(v\otimes x\cdot f)\cdot _2 g .$$
		Thus, $\alpha$ is a $\F[F]$-homomorphism. Since $\alpha$ is clearly bijective, we are done. \end{proof}

	\begin{proof}[Proof of Theorem \ref{algebraic}]
		Since the matrix $A$ involves only a finite number of elements of $G$, we can assume that $G$ is finitely generated. Therefore, there are a finite extension $K$ of $\Q$ and a finite collection $S$ of valuations  of $\mathcal O_K$ such that $\sigma(G)\le \GL_k(\mathcal O_{K,S})$ and $A\in \Mat_{n\times m}(\mathcal O_{K,S}[G])$.

		Since $G$ is a finitely generated sofic group, there exist a finitely generated free group $F$, a normal subgroup $N$ of $F$ and  a family of $F$-sets $\{X_i:\ i\in \N\}$  such that $G\cong F/N$ and $\{X_i:\ i\in \N\}$ is a sofic approximation of $F/N$. 
		
		We fix an infinite collection $\{P_i:\ i\in \N\}$ of maximal ideals of $\mathcal O_{K,S}$ and for each $i\in \N$ we put $\F_i=\mathcal O_{K,S}/P_i$. Then clearly we have that
		\begin{equation}\label{large} \lim_{i\to \infty} |\F_i|\to \infty.\end{equation}
		Denote by $\sigma_i:F\to  \GL_k(\F_i)$  the composition of  the natural map $F\to G$, the map $\sigma:G\to \GL_k(\mathcal O_{K,S})$ and the canonical map $\GL_k(\mathcal O_{K,S})\to \GL_k(\F_i)$. Put  $N_i=\ker \sigma_i$. 
		We consider $Y_i=X_i\times F/N_i$ with diagonal action of $F$. 
		
		\begin{claim}\label{anotherappr}
			The collection $\{Y_i:\ i\in \N\}$ is a sofic  approximation of $F/N$.
		\end{claim}
		\begin{proof}
			Let $w\in F$. Observe that if  $w\in N$, then $\Fix_{X_k}(w)\times F/N_i\subseteq \Fix_{Y_k}(w)$ and if $w\not \in N$, then $\Fix_{Y_k}(w)\subseteq \Fix_{X_k}(w)\times F/N_i$. 
			Thus, since $\{X_i\}$ is a sofic  approximation of $F/N$,  $\{Y_i\}$ is also a sofic  approximation of $F/N$.
		\end{proof}
		Let $B\in \Mat_{n\times m}(\mathcal O_{K,S}[F])$.
		For each $i\in \N$, let $\rho_{Y_i, P_i}(B):\F_i[Y_i]^n\to \F_i[Y_i]^m$ be the map induced by  multiplication by $B$  on the right  side:
		$$(v_1,\ldots,v_n)\rho_{Y_i, P_i}(B)=(v_1,\ldots, v_n)B.$$
		We define  the Sylvester matrix rank function $\rk_{Y_i,P_i}$ over $\mathcal O_{K,S}[F]$ by means of
		$$ \rk_{Y_i,P_i}(B):=\frac{\rk_{\F_i}( \rho_{Y_i, P_i}(B))}{|Y_i|}.$$
		The following result explains why it is more convenient to work with the approximation $\{Y_i\}$ than $\{X_i\}$.
		\begin{claim}\label{finite}
			For each $i\in \N$ and any matrix $B$ over $\mathcal O_{K,S}[F]$, we have that $$\rk_{Y_i,P_i}(\tilde \sigma(B))=k\cdot \rk_{Y_i,P_i}(B).$$
		\end{claim}\label{anotheappr}
		\begin{proof} Let $V=\F_i^k$. It becomes an $\F_i[F]$-module if we define $v\cdot f=v\sigma_i(f)$. Then we obtain that  $F$ acts on $(V\otimes_{\F_i}\F_i[Y_i])_1$ and $(V\otimes_{\F_i}\F_i[Y_i])_2$ as follows.
			$$(v\otimes y)\cdot_1 f=v\otimes (y\cdot f), \ (v\otimes y)\cdot_2 f=(v)\sigma_i(f) \otimes (y\cdot f). $$
			Therefore, if we identify $\F_i^k\otimes_{\F_i} \F_i[Y_i]$ with $\F_i[Y_i]^k$ we obtain that
			$$m\cdot_1 b=(m)\rho_{Y_i, P_i}(b I_k)  \textrm{\ and \ } m\cdot _2 b=(m)\rho_{Y_i, P_i}(\tilde \sigma(b)) \ (m\in \F_i[Y_i]^k, b\in \F_i[F]).$$
			Here $I_k$ denotes the identity $k$ by $k$ matrix. Thus, the claim follows from Lemma \ref{isomtwo}.
		\end{proof}
		Now we compare $ \rk_{Y_i}(B)$ and  $\rk_{Y_i,P_i}(B)$.
		\begin{claim} \label{constantC}
			Let $B\in \Mat_{n\times m}(\mathcal O_{K,S}[F])$. Then there exists a constant $C$ depending only on $B$ such that
			$$|\rk_{Y_i}(B)-\rk_{Y_i,P_i}(B)|\le \frac{C}{\log_2{|\F_i|}} .$$
		\end{claim}
		
		\begin{proof} We prove the claim using results from \cite[Section 8]{Ja19}.   
			
			Let $\alpha\in K$ and let $\alpha_1,\ldots, \alpha_s$ be the roots of the minimal polynomial of $\alpha$ over $\Q$. We put $$\lceil {\alpha}\rceil=\max_i|\alpha_i|.$$
			For any element $b=\sum_{h\in F}a_h h$ ($a_h\in K$) of the group algebra $K[F]$   we put  $$\lceil b\rceil=\sum_{h\in F}\lceil a_h\rceil.$$ 
			We also define  $$  \lceil B \rceil =\max_{j} \sum_i  \lceil b_{ij}\rceil.$$ 
			For each $i\in \N$, let us define
			$$U_i= (\OO_K[Y_i]^m+\OO_K[Y_i]^n \rho_{Y_i}(B))/\OO_K[Y_i]^n \rho_{Y_i}(B)$$ Observe that $$U_i\cong (U_i/U_i^{tor})\oplus U_i^{tor},$$
			where  $ U_i^{tor} $ is the torsion part of the $\OO_K$-module $U_i$.  Therefore,
			\begin{multline*}
				| \rk_{Y_i}(B)-\rk_{Y_i, P_i}(B)|= \rk_{Y_i}(B)-\rk_{Y_i, P_i}(B)=\\ 
				(m- \frac{\dim_{K}(K\otimes _{\OO_K} U_i)}{|Y_i|})-(m-\frac{\dim_{\F_i}(\F_i\otimes _{\OO_K} U_i)}{|Y_i|})=\\ \frac{ \dim_{\F_i}(\F_i\otimes _{\OO_K} U_i^{tor})}{|Y_i|} \le \frac{\log_{|\F_i|} |U_i^{tor}|}{|Y_i|}= \frac{\log_{2} |U_i^{tor}|}{|Y_i|\log_2|\F_i|}.\end{multline*}
			By \cite[Lemmas 8.6 and 8.7]{Ja19}, 
			$$\log_2 |U_i^{tor}|\le m|Y_i| |K:\Q| \log_2\lceil B \rceil.$$
			Therefore, we can take $C=m|K:\Q|\log_2 \lceil B \rceil.$
		\end{proof}
		We are ready to finish the proof of the theorem.  Let $\epsilon>0$.   Let $B$ be a matrix over $\OO_{K,S}[F]$ which maps to  $A$. By (\ref{large}) and Claim \ref{constantC}, 
		there exists $j_1\in \N$  such that for every $i\ge j_1$,
		\begin{equation}
			\label{apprmodp}
			|\rk_{Y_i}(B)-\rk_{Y_i,P_i}(B)|\le \frac \epsilon {4k}  \textrm{\ and\ } |\rk_{Y_i}(\tilde \sigma(B))-\rk_{Y_i,P_i}(\tilde\sigma(B))|\le  \frac \epsilon  4.\end{equation}
		By Claim \ref{anotherappr} and Theorem \ref{soficluck}, there exists $j_2\in \N$  such that for every $i\ge j_2$,
		\begin{equation}\label{apprluck}
			|\rk_G(B)-\rk_{Y_i}(B)|\le  \frac \epsilon  {4k} \textrm{\ and \ } |\rk_G(\tilde \sigma (B))-\rk_{Y_i}(\tilde \sigma( B))|\le \frac \epsilon  4 .\end{equation}
		Therefore putting together Claim \ref{finite}, (\ref{apprmodp}) and (\ref{apprluck}), we obtain that 
		for every $i\ge j_1,j_2$,
		\begin{multline*}|\rk_G(\tilde \sigma (B))-k\cdot \rk_G(B)|\le 
			|\rk_G(\tilde \sigma (B))-\rk_{Y_i}(\tilde \sigma( B))|+\\
			|\rk_{Y_i}(\tilde \sigma( B))-\rk_{Y_i,P_i}(\tilde\sigma(B))|
			+ |\rk_{Y_i,P_i}(\tilde\sigma(B))-k\cdot \rk_{Y_i,P_i}(B) |+\\ k\cdot |\rk_{Y_i,P_i}(B)-\rk_{Y_i}(B)|+k\cdot |\rk_{Y_i}(B)-\rk_G(B)|\le \epsilon.\end{multline*}
		This finishes the proof.
	\end{proof}
	\subsection{Proof of Theorem \ref{twistedsofic}}
	Since the matrix $A$ involves only a finite number of elements of $G$, we can assume that $G$ is finitely generated.  Therefore there are $t_1,\ldots, t_l\in \CC$ such that if we put $R=\overline \Q[t_1,\ldots, t_l]$, then $\sigma(G)\le \GL_k(R)$ and $A$ is a matrix over $R[G]$.
	
	Let $I$ be the kernel of the map $\overline \Q[x_1,\ldots, x_l]\to R$ that sends $x_i$ to $t_i$ ($1\le i\le l$). Then $R\cong \overline \Q[x_1,\ldots, x_l]/I$. If $p=(p_1,\ldots, p_l)\in V(I)$ and $C$ is a matrix over $R[G]$, then we denote by $C (p)$ the image of $C$ after sending $t_i$ to $p_i$  ($1\le i\le l$).  We put $t=(t_1,\ldots, t_l)$. Thus, $A=A(t)$ and $\tilde \sigma (A)=\tilde \sigma (A)(t)$.
	
	\begin{claim} \label{Kazhdan}Let $s_i\in V(I)$ and assume that $\displaystyle \lim_{i\to \infty} s_i=t$. Let $C\in \Mat_{n\times m}(R[G])$ be a matrix over $R[G]$. Then $$\liminf_{i\to \infty} \rk_G(C(s_i))\ge \rk_G(C).$$
	\end{claim}
	\begin{proof} The proof of this claim  is analogous to the proof of Kazhdan's inequality (see \cite[Proposition 10.7]{Ja19b} for details). 
		
		First we construct a spectral measure $\mu$ associated with $C$. Let $T=\rho(C)\rho(C)^*\in \Mat_n(\mathcal B(l^2(G))$. Then $\mu$ is defined as  the measure with support in  $[0, \|T\|]$, whose moments are calculated as
		$$\int x^ld\mu=\Tr_{G}(T^l)=\sum_{i=1}^n\langle ( \mathbf 1_i) T^l,\mathbf 1_i\rangle_{(l^2(G))^n} . $$
		Similarly we define a spectral measure $\mu_i$ associated with $C(s_i)$. 
		
		The condition $\displaystyle \lim_{i\to \infty} s_i=t$ implies that  the meausures $\mu_i$ converge weakly to $\mu$. Therefore, by the Portmanteau theorem (see, for example, \cite[Theorem
		11.1.1]{Du02}),
		$$\limsup_{i\to \infty} \mu_i(\{0\})\le \mu(\{0\}).$$
		However, we have that
		\begin{multline*}
			\mu(\{0\})=\dim_G (\ker (\rho(C)\rho(C)^*))=\dim_G (\ker \rho(C)) \textrm{\ and \ }\\  \mu_i(\{0\})=\dim_G ( \ker (\rho(C(s_i))\rho(C(s_i))^*))=\dim_G (\ker \rho(C(s_i))).\end{multline*}
		Therefore,
		\begin{multline*}\liminf_{i\to \infty} \rk_G(C(s_i))=\liminf_{i\to \infty} (n-\dim_G ( \ker \rho(C(s_i))))=\\ \liminf_{i\to \infty} (n- \mu_i(\{0\}))=n-\limsup_{i\to \infty}\mu_i(\{0\})\ge n- \mu(\{0\})=\rk_G(C).\end{multline*}
	\end{proof}
	
	\begin{claim}\label{le} Let $s\in V(I)$ and $C$   a matrix over $R[G]$. Then $\rk_G (C(s))\le \rk_G (C)$.
	\end{claim}
	\begin{proof} Since $G$ is a finitely generated sofic group, there exist a finitely generated free group $F$, a normal subgroup $N$ of $F$ and  a family of $F$-sets $\{X_i:\ i\in \N\}$  such that $G\cong F/N$ and $\{X_i:\ i\in \N\}$ is a sofic
		approximation of $F/N$. 
		Let $B$ be a matrix over $R[F]$ that maps to  $C$. Then    $B(s)$ maps to  $C(s)$.
		
		Since for each $i\in \N$, the matrix  associated   to  $\rho_{X_i}(B(s))$ (with respect to some bases) is an image  of  the matrix associated to   $\rho_{X_i}(B)$, we obtain that $\rk_{X_i} (B(s))\le \rk_{X_i}(B)$. 
		Now applying Theorem \ref{soficluck}, we conclude that
		$$\rk_G (C(s))=\rk_G (B(s))=\lim_{i\to \infty} \rk_{X_i}(B(s))\le \lim_{i\to \infty} \rk_{X_i}(B)=\rk_G(B)=\rk_G(C). \qedhere$$
	\end{proof}
	
	Now we are ready to finish the proof of the theorem. By \cite[Lemma 3.2]{BG08}, there are $\{s_i\in V(I)\cap \overline \Q\}$ such that  $\displaystyle \lim_{i\to \infty} s_i=t$. By Claims \ref{Kazhdan} and \ref{le}, 
	$$\lim_{i\to \infty} \rk_G(A(s_i))= \rk_G(A) \textrm{\ and \ } \lim_{i\to \infty} \rk_G(\tilde \sigma(A)(s_i))= \rk_G(\tilde\sigma(A)) .$$
	Denote by $\sigma_i:G\to \GL_k(\overline \Q)$ the composition of $\sigma:G\to \GL_k(R)$ and the map $R\to \overline \Q$ sending the $l$-tuple $t$ to the $l$-tuple $s_i$. Then $\tilde \sigma(A)(s_i)=\tilde{\sigma_i}(A(s_i))$. 
	Applying Theorem \ref{algebraic}, we obtain that 
	$$\rk_G(\tilde \sigma(A)(s_i))=
	\rk_G(\tilde{\sigma_i}(A(s_i))=
	k\cdot \rk_G(A(s_i)).
	$$ 
	Thus,  $\rk_G(\tilde\sigma(A))=k\cdot \rk_G(A)$. \qedsymbol
	
	\begin{rem} In  \cite[Theorem 1.1]{Ja19} it was proved that for a sofic group G, the strong Atiyah conjecture   over $\overline \Q$ implies the strong Atiyah conjecture    over $\CC$. We would like to notice that
		the previous argument provides an alternative way to obtain this result from the sofic L\"uck aproximation.
		
	\end{rem}


\begin{thebibliography}{33}
		

        \bibitem{BG08}  E. Breuillard, T.  Gelander, Uniform independence in linear groups, Invent. Math. 173  (2008), 225--263.
		
		\bibitem {Du02} R. M. Dudley,   Real analysis and probability.   Revised reprint of the 1989 original. Cambridge Studies in Advanced Mathematics, 74. Cambridge University Press, Cambridge, 2002.
		
		\bibitem{Ja19}  A. Jaikin-Zapirain,  The base change in the Atiyah and the L\"uck approximation conjectures. Geom. Funct. Anal. 29 (2019),   464--538. 
		
		\bibitem{Ja19b}  A. Jaikin-Zapirain, $L^2$-Betti numbers and their analogues in positive characteristic. Groups St Andrews 2017 in Birmingham, 346--405, London Math. Soc. Lecture Note Ser., 455, Cambridge Univ. Press, Cambridge, 2019.
		
		\bibitem{JL20}  A. Jaikin-Zapirain, D. L\'opez-\'Alvarez,   The strong Atiyah and L\"uck approximation conjectures for one-relator groups. Math. Ann. 376 (2020),  1741--1793. 
		
		\bibitem{KS21}  D. Kielak, B. Sun,  Agrarian and $l^2$-Betti numbers of locally indicable groups, with a twist, arXiv:2112.07394, accepted to Math. Ann..
		
		\bibitem{Lu94}  W. L\"uck, 
		Approximating $L^2$-invariants by their finite-dimensional analogues,
		Geom. Funct. Anal. 4 (1994),   455--481. 
		
		\bibitem{Lu18}  W. L\"uck,  Twisting $L^2$-invariants with finite-dimensional representations. J. Topol. Anal. 10 (2018),   723--816.
		

	\end{thebibliography}
\end{document}